\documentclass[runningheads,a4paper]{llncs}

\usepackage{amssymb}
\usepackage{graphicx}
\graphicspath{{images/},{prebuiltimages/}}
\usepackage{url}
\usepackage[utf8]{inputenc}

\usepackage{subfigure}
\usepackage{color}
\usepackage{algpseudocode}
\usepackage{mystyle}
\usepackage{title_ssvm}
\usepackage{enumitem}
\usepackage{soul}
\usepackage{ulem}

\newcommand{\Tik}{\star}
\newcommand{\LS}{\star}
\newcommand{\CLS}{\star}
\newcommand{\HT}{\star}
\newcommand{\LO}{\star}

\newcommand{\ST}{\star}
\newcommand{\TV}{\star}

\newcommand{\Uf}{U_f}

\newcommand{\uf}{u_f}
\newcommand{\ufO}{u_{f_0}}
\newcommand{\Mmf}{\Mm_f}
\newcommand{\MmfO}{\Mm_{f_0}}
\newcommand{\Jf}{J_f}

\newcommand{\duf}{\tilde{u}_f}
\newcommand{\dufO}{\tilde{u}_{f_0}}
\newcommand{\dMmf}{\tilde{\Mm}_f}

\newcommand{\dJf}{\tilde{J}_f}

\newlength{\savedbelowdisplayskip}
\newlength{\savedbelowdisplayshortskip}
\newlength{\savedabovedisplayskip}
\newlength{\savedabovedisplayshortskip}
\newcommand{\changemargin}[3]{%
  \setlength{\savedbelowdisplayskip}{\belowdisplayskip}%
  \setlength{\savedbelowdisplayshortskip}{\belowdisplayshortskip}%
  \setlength{\savedabovedisplayskip}{\abovedisplayskip}%
  \setlength{\savedabovedisplayshortskip}{\abovedisplayshortskip}%
  \addtolength{\belowdisplayskip}{#2}\addtolength{\belowdisplayshortskip}{#2}%
  \addtolength{\abovedisplayskip}{#1}\addtolength{\abovedisplayshortskip}{#1}%
  #3%
  \setlength{\belowdisplayskip}{\savedbelowdisplayskip}%
  \setlength{\belowdisplayshortskip}{\savedbelowdisplayshortskip}%
  \setlength{\abovedisplayskip}{\savedabovedisplayskip}%
  \setlength{\abovedisplayshortskip}{\savedabovedisplayshortskip}%
}

\usepackage[textwidth=4cm, textsize=footnotesize]{todonotes}
\usepackage{color}

\usepackage{url}
\urldef{\mailsa}\path|{charles-alban.deledalle, nicolas.papadakis}@math.u-bordeaux.fr|
\urldef{\mailsb}\path|joseph.salmon@telecom-paristech.fr|

\usepackage{ifthen}
\newboolean{arxiv}

\begin{document}
\sloppy
\mainmatter

\setboolean{arxiv}{false}

%
%
\def\SSVM15SubNumber{007}

%
%
\title{\ssvmtitle
}

%
%
%
\titlerunning{\ssvmtitleshort}
\author{Charles-Alban Deledalle\inst{1,2}, Nicolas Papadakis\inst{1,2}, Joseph Salmon\inst{3}}

\authorrunning{Charles-Alban Deledalle \and Nicolas Papadakis \and Joseph Salmon}

\institute{%
$^1$ Univ. Bordeaux, IMB, UMR 5251, F-33400 Talence, France.\\
$^2$ CNRS, IMB, UMR 5251, F-33400 Talence, France.\\ \mailsa\\
$^3$ Institut Mines-Télécom, Télécom ParisTech,
CNRS LTCI, Paris, France.\\ \mailsb}

\maketitle

\begin{abstract}
Bias in image restoration algorithms
can hamper further analysis,
typically when the intensities have a physical meaning of interest, e.g.,
in medical imaging.
We propose to suppress a part of the bias
-- {\it the method bias} -- while leaving unchanged
the other unavoidable part -- {\it the model bias}.
Our debiasing technique can be used for any locally affine estimator
including $\ell_1$ regularization, anisotropic total-variation and some
nonlocal filters.
\end{abstract}

\section{Introduction}

Restoration of an image of interest from its single noisy degraded observation
necessarily requires imposing some regularity or {\it prior} on the solution.
Being often only crude approximations of the true underlying signal of interest,
such techniques always introduce a bias towards the {\it prior}.
However, in general, this is not the only source of bias.
In many cases, even though the model was perfectly accurate,
the method would remain biased.
This part of the bias often emerges from technical reasons,
e.g., when approaching an NP-hard problem by an easier one
(typically, using the $\ell_1$ convex relaxation of an $\ell_0$ pseudo-norm).

It is well known that reducing bias is not always favorable in terms
of mean square error because of the so-called bias-variance trade-off.
It is important to highlight that a debiasing procedure is expected to
re-inject part of the variance, therefore increasing the residual noise.
Hence, the mean square error is not always expected to be improved by such techniques.
Debiasing is nevertheless essential
in applications where the image intensities have a physical sense and critical decisions
are taken from their values. For instance, the authors of \cite{de2013extended} suggest using
image restoration techniques to estimate a temperature map within a tumor tissue
for real time automatic surgical intervention.
In such applications, it is so crucial that the estimated temperature is not biased.
A remaining residual noise is indeed favorable compared to
an uncontrolled bias.

We introduce a debiasing technique that suppresses the extra bias -- {\it the method bias} --
emerging from the choice of the method and leave unchanged the bias that is
due to the unavoidable choice of the model
-- {\it the model bias}.
To that end, we rely on
the notion of model subspace
essential to carefully define different notions of bias.
This leads to a mathematical definition of
debiasing for any locally affine
estimators that respect some mild assumptions.

Interestingly, our debiasing definition for the
$\ell_1$ synthesis
(also known as LASSO \cite{tibshirani1996regression} or Basis Pursuit \cite{Chen_Donoho_Saunders98})
recovers a well known
debiasing scheme called refitting that goes back to the ``Hybrid LASSO'' \cite{efron2004least}
(see \cite{Lederer13} for more details).

For the $\ell_1$ analysis \cite{elad2007analysis},
including the $\ell_1$ synthesis but also the anisotropic total-variation
\cite{rudin1992nonlinear},
we show that debiasing can be performed with the same complexity as
the primal-dual algorithm of \cite{CP} producing the biased estimate.

In other cases, e.g.,
for an affine version of the popular nonlocal-means
\cite{buades2005nlmeans},
we introduce an iterative scheme that requires
only a few run of an algorithm of the same complexity as
the original one producing the biased estimate.


\section{Background}

We consider observing $f = f_0 + w \in \RR^P$ a corrupted linear observation
of an unknown signal $u_0 \in \RR^N$  such that $f_0 = \Phi u_0$
where $\Phi \in \RR^{N \times P}$ is a linear operator
and $w$ is a random vector modeling the noise fluctuations. We assume that $\EE[w] = 0$
where $\EE$ is the expectation operator.
The linear operator $\Phi$ is a degrading operator typically
with $P\leq N$ and with
a non-empty kernel encoding some information loss
such that the problem becomes ill-posed.

We focus on estimating the unknown signal $u_0$.
Due to the ill-posedness of the observation model,
we consider variational approaches that attempt to recover $u_0$ from
the single observation $f$ as a solution of the optimization problem
\changemargin{-2pt}{-2pt}{%
\begin{equation}\label{eq:variationnal_model}
  \uf^\star \in \uargmin{u \in \RR^N} E(u, f)~.
\end{equation}}%
where $E : \RR^N \times \RR^P \to \RR$ is assumed
to have at least one minimum.
The objective $E$ is typically chosen to promote some
structure, e.g., smoothness, piece-wise constantness, sparsity, etc.,
that is captured by the so-called {\it model subspace} $\Mmf^*$.
Providing $\uf^\star$ is uniquely defined and differentiable at $f$,
we define $\Mmf^\star \subseteq{\RR^N}$ as the tangent affine subspace at $f$ of
the mapping $f \mapsto \uf^\star$, i.e.,
\changemargin{-6pt}{-3pt}{%
\begin{equation}
  \Mmf^\star = \uf^\star + \Ima[ \Jf^\star ] = \enscond{u \!\in\! \RR^N}{\exists z \!\in\! \RR^P, u = \uf^\star + \Jf^\star z}
  \;\;\text{with}\;\;
  \Jf^\star = \left.\frac{\partial \uf^\star}{\partial f}\right|_f
\end{equation}}%
where $\Jf^\star$ is the Jacobian operator at $f$ of the mapping $f \!\mapsto\! \uf^\star$
(see \cite{vaiter2014model} for an alternative but related definition
of model subspace).
When $\uf^\star \!\in\! \Ima[ \Jf^\star ]$, the model subspace
restricts to the linear vector subspace $\Mmf^\star \!=\! \Ima[ \Jf^\star ]$.
In the rest of the paper, $\uf^\star$ is assumed to be
differentiable at $f_0$ and for almost all $f$.

\begin{exmp}
The least square estimator
constrained to the affine subspace $C = b + \Ima[A]$,
$b \in \RR^N$ and $A \in \RR^{N \times Q}$,
is a particular instance of \eqref{eq:variationnal_model} where
\changemargin{-2pt}{-1pt}{%
\begin{equation}\label{eq:constrained_least_square_energy}
  E(u, f) = \norm{\Phi u - f}^2 + \iota_C(u)
\end{equation}}%
and for any set $C$, $\iota_C$ is its indicator function:
$\iota_C(u) = 0$ if $u \in C$, $+\infty$ otherwise.
The solution of minimum Euclidean norm is unique and given by
\changemargin{-1pt}{-4pt}{%
\begin{equation}\label{eq:contrained_least_square}
  \uf^\CLS = b + A (\Phi A)^+  (f - \Phi b)
\end{equation}}%
where for a matrix $M$, $M^+$ is its Moore-Penrose pseudo-inverse.
The affine constrained least square
restricts the solution $\uf^\CLS$ to the
affine model subspace $\Mmf^\CLS \!=\! b + \Ima[ A (\Phi A)^t ]$
(as $\Ima[M^+] \!=\! \Ima[M^t]$).
Taking $C \!=\! \RR^N$ with for instance $Q\!=\!N$, $A = \Id$ and $b = 0$, leads to an
unconstrained solution
$
  \uf^\LS \!=\! \Phi^+ f
$
whose model subspace is $\Mmf^\LS \!=\! \Ima[ \Phi^t ]$ reducing to $\RR^N$ when $\Phi$ has full column rank.
\end{exmp}

\begin{exmp}
The Tikhonov regularization (or Ridge regression)
\cite{tikhonov43,hoerl1970ridge}
is another instance of \eqref{eq:variationnal_model} where,
for some parameter $\lambda > 0$ and matrix $\Gamma \in \RR^{L \times N}$,
\begin{equation}\label{eq:tikhonov}
  E(u, f) = \frac{1}{2} \norm{\Phi u - f}^2 + \frac{\lambda}{2} \norm{\Gamma u}^2~.
\end{equation}
Provided $\Ker \Phi \cap \Ker \Gamma \!=\! \{ 0 \}$,
$\uf^\Tik$ is uniquely defined as
$
  \uf^\Tik = (\Phi^t \Phi + \lambda \Gamma^t \Gamma)^{-1} \Phi^t  f
$
which has a linear model subspace given by $\Mmf^\Tik = \Ima[ \Phi^t ]$.
\end{exmp}

\begin{exmp}
The hard thresholding \cite{donoho1994ideal},
used when $\Phi = \Id$ and $f_0$ is supposed to be sparse,
is a solution of \eqref{eq:variationnal_model} where,
for some parameter $\lambda > 0$,
\begin{equation}
  E(u, f) = \frac{1}{2} \norm{u - f}^2 + \frac{\lambda^2}{2} \norm{u}_0~,
\end{equation}
where $\norm{u}_0 = \#\enscond{i \in [P]}{u_i \ne 0}$ counts the number of non-zero entries of $u$
and $[P] = \{1, \ldots, P\}$.
The hard thresholding operation writes
\begin{equation}
  (\uf^\HT)_{\Ii_f} = f_{\Ii_f} \qandq (\uf^\HT)_{\Ii_f^c} = 0
\end{equation}
where $\Ii_f \!=\! \enscond{i \!\in\! [P]}{|f_i| \!>\! \lambda}$ is
the support of $\uf^\HT$,
$\Ii_f^c$ is the complement of $\Ii_f$ on $[P]$,
and for any vector $v$,
$v_{\Ii_f}$ is the sub-vector whose elements are indexed by $\Ii_f$.
As $\uf^\HT$ is piece-wise differentiable,
its model subspace is only defined for almost all $f$ as
$\Mmf^\HT \!=\! \smallenscond{u \!\in\! \RR^N}{u_{\Ii^c_f} \!=\! 0} \!=\! \Ima[\Id_{\Ii_f}]$,
where for any matrix $M$,
$M_{\Ii_f}$ is the sub-matrix
whose columns are indexed by $\Ii_f$.
Note that $\Id_{\Ii_f} \!\in\! \RR^{N \times \# \Ii_f}$.%
\end{exmp}

\begin{exmp}
The soft thresholding \cite{donoho1994ideal},
used when $\Phi = \Id$ and $f_0$ is supposed to be sparse, is
another particular solution of \eqref{eq:variationnal_model} where
\begin{equation}\label{eq:ST}
  E(u, f) = \frac{1}{2} \norm{u - f}^2 + \lambda \norm{u}_1~,
\end{equation}
with $\norm{u}_1 = \sum_i |u_i|$ the $\ell_1$ norm of $u$.
The soft thresholding operation writes
\begin{equation}
  (\uf^\ST)_{\Ii_f} = f_{\Ii_f} - \lambda \sign(f_{\Ii_f}) \qandq (\uf^\ST)_{\Ii_f^c} = 0~,
\end{equation}
where $\Ii_f$ is defined as above, and,
as for the hard thresholding:
$\Mmf^\ST = \Ima[\Id_{\Ii_f}]$.
\end{exmp}

\section{Bias of reconstruction algorithms}

Due to the ill-posedness of our observation model and without any assumptions on $u_0$,
one cannot ensure the noise variance to be reduced while
keeping the solution $\uf^\star$ unbiased.
Recall that the statistical bias is defined as the difference
\begin{equation}
  \text{Statistical bias} = \EE[\uf^\star] - u_0~.
\end{equation}
An estimator is said unbiased when its statistical bias vanishes.
Unfortunately the statistical bias is difficult to manipulate when
$f \mapsto \uf^\star$ is non linear.
We therefore restrict to a definition of bias at $f_0 = \Phi u_0$ as the error
$\ufO^\star - u_0$.
Note that when $f \mapsto \uf^\star$  is affine, both definitions match (the expectation being linear).
Most methods are biased since, without assumptions, $u_0$ cannot be guaranteed
to be in complete accordance with
the model subspace, i.e., $u_0 \notin \MmfO^\star$.
It is then important to distinguish techniques that are only biased
due to a problem of modeling to the ones that are biased
due to the method. We then define the model bias and the method bias as
the quantities
\begin{equation}
  \ufO^\star - u_0
  =
  \underbrace{\ufO^\star - \Pi_{\MmfO^\star}(u_0)}_{\text{Method bias}}
  -
  \underbrace{\Pi_{(\MmfO^\star)^\bot}(u_0)}_{\text{Model bias}}~,
\end{equation}
where for any set $S$, $\Pi_S$ denotes the orthogonal projection on $S$
and $S^\bot$ denotes its orthogonal set.
We now define a methodically unbiased estimator as follows.

\begin{defn}
An estimator $\uf^\star$ is \underline{\smash{methodically unbiased}} if
\begin{equation*}
  \forall u_0 \in \RR^N, \quad
  \ufO^\star = \Pi_{\MmfO^\star}(u_0)
\end{equation*}
\end{defn}

We also define the weaker concept of weakly unbiased estimator as follows.

\begin{defn}
An estimator $\uf^\star$ is \underline{\smash{weakly unbiased}} if
\begin{equation*}
  \forall u_0 \in \MmfO^\star, \quad
  \ufO^\star = u_0.
\end{equation*}
The quantity $\ufO^\star - u_0$ for $u_0 \in \MmfO^\star$ is
called the weak bias of $\uf^\star$ at $u_0$.
\end{defn}

Remark that a methodically unbiased estimator is also weakly unbiased.

\paragraph{Examples.} The unconstrained least-square estimator is
methodically unbiased since
$\ufO^\LS \!=\! \Phi^+ f_0 \!=\! \Phi^+ \Phi u_0 \!=\! \Pi_{\Ima[\Phi^t]}(u_0) \!=\! \Pi_{\MmfO^\star}(u_0)$.
Moreover, being linear,
it becomes statistically unbiased
whenever $\Phi$ has full column rank since $\Phi^+ \Phi \!=\! \Id$.
However the constrained least-square estimator is only weakly unbiased:
its methodical bias only vanishes
when $u_0 \!\in\! \MmfO^\CLS$, i.e.,
when there exists $t_0 \in \RR^Q$ such that $u_0 \!=\! b + A (\Phi A)^t t_0$.
The hard thresholding is also methodically unbiased
remarking that $\ufO^\star$ is
the orthogonal projection on $\MmfO^\star \!=\! \Ima[\Id_{\Ii_{f_0}}]$.
Unlike the unconstrained least-square estimator,
Tikhonov regularization has a non zero weak bias.
The soft thresholding is also known to be biased \cite{fan2001variable}
and its weak bias is given by $-\lambda \Id_{\Ii_{f_0}} \sign(f_0)_{\Ii_{f_0}}$.
Often, estimators are said to be unbiased when
they are actually only weakly unbiased.

\section{Definitions of debiasing}

Given an estimate $\uf^\star$ of $u_0$,
we define a debiasing of $\uf^\star$ as follows.

\begin{defn}\label{eq:def_debiasing}
  An estimator $\duf^\star$ of $u_0$ is a \underline{\smash{weak debiasing}} of $\uf^\star$
  if it is weakly unbiased and $\dMmf^\star = \Mmf^\star$ for almost all $f$,
  with $\dMmf^\star$ the model subspace of $\duf^\star$ at $f$.
  Moreover, it is a \underline{\smash{methodical debiasing}} if it is also methodically
  unbiased.
\end{defn}


\paragraph{Examples.}
The unconstrained least square estimator is a methodical debiasing of
the Tikhonov regularization,
since it is a methodically unbiased estimator of $u_0$
and they share the same model subspace.
The hard thresholding is a methodical debiasing of
the soft thresholding, for the same reasons.
\medskip

A good candidate for debiasing $\uf^\star$ is the constraint least squares
on $\Mmf^\star$:
\changemargin{-1pt}{-4pt}{%
\begin{equation}\label{eq:debiasing}
  \duf^\star
  = \uf^\star + \Uf^\star (\Phi \Uf^\star)^+ (f - \Phi \uf^\star)
  \in \uargmin{u \in \Mmf^\star} \norm{\Phi u - f}^2
\end{equation}}%
where $\Uf^\star \!\in \! \RR^{N \times n}$ with $n \!=\! \rank[\Jf^\star]$
is a matrix whose columns form a basis of $\Ima[\Jf^\star]$.
Let $V_f^\star \!\in \! \RR^{n \times P}$ be a matrix such that $J_f^\star = \Uf^\star V_f^\star$.
The following theorem shows that under mild assumptions this choice corresponds to
a debiasing of $\uf^\star$.



\begin{thm}\label{lem:tilde_unbiased}
  Assume that $f \mapsto \uf^\star$ is locally affine for almost all $f$
  and that $\Phi$ is invertible on $\Mmf^\star$.
  Then $\duf^\star$ defined in Eq.~\eqref{eq:debiasing}
  is a weak debiasing of $\uf^\star$.
\end{thm}

\begin{proof}
  Since $f \mapsto \uf^\star$ is locally affine,
  $f \mapsto \Uf^\star$ can be chosen locally constant.
  Deriving \eqref{eq:debiasing} for almost all $f$
  leads to the Jacobian $\dJf^\star$ of $\duf^\star$ given by
  \changemargin{-4pt}{0pt}{%
  \begin{eqnarray}
  \dJf^\star
  &=& \pd{\duf^\star}{f} 
  = \pd{\uf^\star}{f} + \Uf^\star (\Phi \Uf^\star)^+ \left(\pd{f}{f} - \Phi \pd{\uf^\star}{f}\right)
  = \Jf^\star + \Uf^\star (\Phi \Uf^\star)^+ (\Id - \Phi \Jf^\star) \nonumber\\
  &=& \Uf^\star V_f^\star + \Uf^\star (\Phi \Uf^\star)^+ (\Id - \Phi \Uf^\star V_f^\star)
  = \Uf^\star (\Phi \Uf^\star)^+
  ~,
  \end{eqnarray}}%
  since $\Phi \Uf^\star$ has full column rank due to the assumption that
  $\Phi$ is invertible on $\Mmf^\star$%
  \ifthenelse{\boolean{arxiv}}{ (see Appendix \ref{sec:details_proof1}).}{.}
  It follows that
  \changemargin{-2pt}{-2pt}{%
   \begin{eqnarray}
    \dMmf^\star
    &=& \duf^\star + \Ima[\dJf^\star]
    = \uf^\star + \Uf^\star (\Phi \Uf^\star)^+ (f - \Phi \uf^\star) +
    \Ima[\Uf^\star (\Phi \Uf^\star)^+]\\
    &=& \uf^\star + \Ima[\Uf^\star (\Phi \Uf^\star)^+]
    = \uf^\star + \Ima[\Uf^\star] = \Mmf^\star~,
  \end{eqnarray}}%
  since $\Phi \Uf^\star$ has full column rank.
  Moreover, for any $u_0 \!\in\! \MmfO^\star$,
  the equation $\Phi u \!=\! f_0$ has a unique solution $u \!=\! u_0$ in $\MmfO^\star$
  since $\Phi$ is invertible on $\MmfO^\star$.
  Hence, $\dufO^\star \!=\! u_0$ is the unique solution of \eqref{eq:debiasing},
  which concludes the proof.
  \hfill $\square$
\end{proof}

The next proposition shows that the condition
``$\Phi$ invertible on $\Mmf^\star$'' can be dropped
when looking at $u_f^\star$ and $\duf^\star$ through $\Phi$.
The debiasing becomes furthermore methodical.

\begin{prop}
  Assume $f \mapsto \uf^\star$ is locally affine for almost all $f$.
  Taking $\tilde{u}_f^\star$ defined in Eq.~\eqref{eq:debiasing},
  then the predictor $\Phi \tilde{u}_f^\star$  of $f_0 \!=\! \Phi u_0$
  is equal to $\Pi_{\Phi \Mm_f^\star}(f)$
  and is a methodical debiasing of $\Phi u_f^\star$.
\end{prop}

\begin{proof}
  Since
  $\Phi \Uf^\star (\Phi \Uf^\star)^+\!=\!\Pi_{\Ima[\Phi U_f^\star]}$,
  we have
  $\Phi \tilde{u}^\star_f \!=\! \Pi_{\Phi \Mm_f^\star}(f)$.
  As the orthogonal projector on its own model space,
  it is methodically unbiased.
  Moreover $\Ima[\Pi_{\Ima[\Phi U_f^\star]}] \!=\! \Ima[\Phi U_f^\star]$, hence
  $\Phi \tilde{u}_f^\star$ and $\Phi u^\star_f$ share the same model subspace.
  \hfill $\square$
\end{proof}


\begin{rem}
As an immediate consequence, the debiasing
of any locally affine denoising algorithm is a methodical debiasing,
since $\Phi \tilde{u}_f^\star = \tilde{u}_f^\star$.
\end{rem}


\bigskip
We focus in the next sections on the debiasing of
estimators without explicit expression for $\Mmf^\star$,
meaning that Eq.~\eqref{eq:debiasing} cannot be used directly.
We first introduce an algorithm for the case of $\ell_1$ analysis
relying on the computation of the directional derivative $J^\star_f f$.
We propose next a general approach,
applied to an affine nonlocal estimator,
that requires $J^\star_f \delta$
for randomized directions $\delta$.


\section{Debiasing the $\ell_1$ analysis minimization}

From now on, the dependency of all quantities
with respect to the observation $f$ will be dropped for the sake of simplicity.
Given a linear operator $\Gamma \!\in\! \RR^{L \times N}$,
the $\ell_1$ analysis minimization
reads, for $\lambda > 0$, as
\changemargin{-2pt}{-2pt}{%
\begin{equation}\label{eq:l1analysis}
  E(u, f) = \frac{1}{2} \norm{\Phi u - f}^2 + \lambda \norm{\Gamma u}_1~.
\end{equation}}%
Provided $\Ker \Phi \cap \Ker \Gamma \!\!=\!\! \{ 0 \}$,
there exists a solution given implicitly, see \cite{vaiter2013local}, as
\changemargin{-2pt}{-2pt}{%
\begin{equation}\label{eq:l1analysis_solution}
  u^\star = U(\Phi U)^+ f - \lambda U(U^t \Phi^t \Phi U)^{-1} U^t (\Gamma^t)_\Ii s_\Ii
\end{equation}}%
for almost all $f$ and
where $\Ii \!=\! \enscond{i}{(\Gamma u^\star)_i \ne 0} \!\subseteq\! [L] \!=\!\{1, \ldots, L\}$
is called the co-support of the solution,
$s \!=\! \sign(\Gamma u^\star)$,
$U\!=\!U_f^\star$ is a matrix whose columns form a basis of $\Ker[\llic{\Gamma}]$
and $\Phi U$ has full column rank.
Note that $s_\Ii$ and $U$ are locally constant almost everywhere since the co-support
is
stable
with respect to
small perturbations \cite{vaiter2013local}.
It then follows that the model subspace is implicitly
defined as $\Mm^\star \!=\! \Ima[U] \!=\! \Ker[\llic{\Gamma}]$,
and so, the $\ell_1$ analysis minimization
suffers from a
weak bias equal to
$-\lambda U(U^t \Phi^t \Phi U)^{-1} U^t (\Gamma^t)_\Ii s_\Ii$.
Given that $u^\star \in \Ima[U]$ and it is locally affine,
its weak debiased solution is defined for almost all $f$ as
\changemargin{-2pt}{-2pt}{%
\begin{equation}\label{eq:debiasing_al1}
  \tilde{u}^\star
  = U (\Phi U)^+ f~.
\end{equation}}%

\paragraph{The $\ell_1$ synthesis} \cite{tibshirani1996regression,donoho1994ideal} consists in
taking $\Gamma\!=\!\Id$, hence $U\!=\!\Id_\Ii$, so \eqref{eq:l1analysis_solution} becomes
\changemargin{-2pt}{-2pt}{%
\begin{equation}
  u^\LO_{\Ii} = (\Phi_{\Ii})^+ f - \lambda ((\Phi_{\Ii})^t \Phi_{\Ii})^{-1} s_{\Ii}
  \qandq
  u^\LO_{(\Ii)^c} = 0~.
\end{equation}}%
Its model subspace is implicitly defined as
$\Mm^\LO \!=\! \Ima[\Id_{\Ii}]$,
its weak bias is $-\lambda \Id_{\Ii}((\Phi_\Ii)^t \Phi_{\Ii})^{-1} s_\Ii$
and its weak debiasing is $\tilde{u}^\star = \Id_\Ii (\Phi_\Ii)^+ f$.
Subsequently, taking $\Phi = \Id$ leads to the soft-thresholding presented earlier.

\paragraph{The anisotropic Total-Variation (TV)} \cite{rudin1992nonlinear}
is a particular instance of \eqref{eq:l1analysis}
where $u_0 \in \RR^N$ can be identified to a $d$-dimensional discrete signal,
for which $\Gamma \! \in \! \RR^{L \times N}$, with $L\!=\!dN$, is
the concatenation of the discrete gradient operators
in each canonical directions.
In this case $\Ii$ is the set of indexes where the solution has discontinuities (non-null gradients)
and $\Mm^\TV$ is the space of piece-wise constant signals sharing the same discontinuities
as the
solution.
Its weak bias reveals a loss of contrast: a shift of intensity on each piece
depending on its surrounding and the ratio between its perimeter and its area,
as shown, e.g., in \cite{strong2003edge}.
Note that the so-called {\it staircasing} effect of TV regularization is encoded
in our framework as a model bias, and
is therefore not reduced by our debiasing technique.
Strategies devoted to the reduction of this effect
have been studied in, \!e.g., \!\cite{louchet2011total}.
\bigskip

Since in general $u^\star$ has no explicit solutions,
it is usually estimated thanks to an iterative algorithm
that can be expressed as a sequence $u^k$
converging to $u^\star$.
The question we address is how to compute $\tilde{u}^\star$ in practice, i.e.,
to evaluate Eq.~\eqref{eq:debiasing_al1}, or more precisely, how to jointly build
a sequence $\tilde{u}^k$
converging to $\tilde u^\star$.

We propose a technique that relies on the observation that,
given \eqref{eq:l1analysis_solution}, for almost all $f$,
the Jacobian $J^\star$ of $u^\star$ at $f$ applied to $f$, leads to Eq.~\eqref{eq:debiasing_al1}, i.e.,
\changemargin{-2pt}{-2pt}{%
\begin{equation}
  J^\star [f] = U(\Phi U)^+ f = \tilde{u}^\star
\end{equation}}%
since $U$ and $s_\Ii$ are locally constant
\cite{vaiter2013local}.
We so define a sequence $\tilde{u}^k$
which is, up to a slight modification,
the closed-form derivation of the primal-dual sequence $u^k$ of \cite{CP}.
Most importantly, we provide a proof of its convergence towards $\tilde{u}^\star$.

Note that other debiasing techniques could be
employed for the $\ell_1$ analysis, e.g., using
iterative hard-thresholding \cite{herrity2006sparse,blumensath2008iterative},
refitting techniques \cite{efron2004least,Lederer13},
post-refinement techniques based an Bregman divergences
and nonlinear inverse scale spaces
\cite{osher2005iterative,burger2006nonlinear,xu2007iterative}
or with ideal spectral filtering in the analysis sense
\cite{gilboa2014total}.
\ifthenelse{\boolean{arxiv}}{%
We invite the interested reader to have a look at Appendix \ref{sec:comp_debiaising}
for a study of the advantages and limitations of such methods
compared to ours.
}

\subsection{Primal-dual algorithm}

Before stating our main result, let us recall some of the properties of
primal-dual techniques. Dualizing the $\ell_1$ analysis norm $u \mapsto \lambda \norm{\Gamma u}_1$,
the primal problem can be reformulated as the following saddle-point problem
\changemargin{-2pt}{-2pt}{%
\begin{equation}\label{eq:l1analysis_dual}
  z^\star = \uargmax{z \in \RR^L} \min_{u \in \RR^N}  \frac{1}{2} \|\Phi u-f\|^2+\left< \Gamma u, z\right> - \iota_{B_\lambda}(z)
\end{equation}}%
where $z^\star \in \RR^L$ is the dual variable, and $B_\lambda=\enscond{z}{\norm{z}_\infty\leq \lambda}$ is the $\ell_\infty$ ball.


\paragraph{First order primal-dual optimization.}

Taking $\sigma\tau<\frac1{\|\Gamma\|_2^2}$, $\theta\in[0,1]$ and initializing (for instance,) $u^0=v^0=0 \in \RR^N$, $z^0=0 \in \RR^L$, the primal-dual  algorithm of \cite{CP} applied to problem \eqref{eq:l1analysis_dual} reads
\changemargin{-2pt}{-2pt}{%
\begin{equation}\label{algo:normal}\left\{\begin{array}{ll}
z ^{k+1}&=\Pi_{B_\lambda}(z^k+\sigma \Gamma v^k),\\
u^{k+1}&=(\Id+\tau\Phi^t\Phi) ^{-1}\left(u^k+\tau( \Phi^tf-\Gamma^t(z^{k+1}))\right),\\
v^{k+1}&=u^{k+1}+\theta(u^{k+1}-u ^k),
\end{array}\right.\end{equation}}%
where the projection of $z$ over $B_\lambda$ is done component-wise as
\changemargin{-2pt}{-2pt}{%
\begin{equation}
  \Pi_{B_\lambda}(z)_i=
  \left\{
    \begin{array}{ll}
      z_i & \text{if} \quad |z_i|\leq \lambda,\\
      \lambda\,  \sign(z_i) \quad \; & \otherwise.
    \end{array}
  \right.
\end{equation}}%
The primal-dual sequence $u^k$ converges to
a solution $u^\star$ of \eqref{eq:l1analysis} \cite{CP}.
We assumed here that $u^\star$ verifies \eqref{eq:l1analysis_solution}
with $\Phi U$ full-column rank.
This could be enforced as shown in \cite{vaiter2013local},
but it did not seem to be necessary in our experiments.

\subsection{Debiasing algorithm}

As pointed out earlier, the debiasing of $u^\star$ consists in applying the
Jacobian matrix $J^\star$ at $f$ to $f$ itself.
This idea leads to the proposed debiasing algorithm that constructs a sequence of debiased
iterates from the original biased primal-dual sequence with initialization
$\tilde u^0=\tilde v^0=0 \in \RR^N$, $\tilde z^0=0 \in \RR^L$
as follows
\begin{eqnarray}\label{algo:diff}
&&
\left\{\begin{array}{ll}
\tilde z^{k+1}&=\Pi_{z^k+\sigma \Gamma v^k}(\tilde z^k+\sigma \Gamma \tilde v^k),\\
\tilde u^{k+1}&=(\Id+\tau \Phi^t\Phi)^{-1}\left(\tilde u^k+\tau(\Phi^t f-\Gamma^t\tilde z^{k+1})\right),\\
\tilde v^{k+1}&=\tilde u^{k+1}+\theta(\tilde u^{k+1}-\tilde u ^k),
\end{array}\right.\\
\whereq &&
\Pi_{z^k+\sigma \Gamma v^k}(\tilde z_i)=
\left\{
\begin{array}{ll}
  \tilde z_i \quad \; & \text{if} \quad |z^k+\sigma \Gamma v^k|^{}_i\leq \lambda+\beta,\\
  0 & \otherwise.
\end{array}
\right. \nonumber
\end{eqnarray}
with $\beta \geq 0$.
Note that when $\beta=0$,
deriving $z^k$, $u^k$ and $v^k$ for almost all $f$ at $f$ in the direction $f$ using the chain rule
leads to the sequences $\tilde z^k$, $\tilde u^k$ and $\tilde v^k$ respectively
(see also \cite{deledalle2014stein}).
However, as shown in Theorem \ref{prop:convergence}, it is important to choose $\beta > 0$ to
guarantee the convergence of the sequence\footnote{In practice, $\beta$ can be chosen
as the smallest positive floating number.}.

\begin{thm}\label{prop:convergence}
Let $\alpha>0$ be the minimum non zero value%
\footnote{If $|\Gamma u^\star|_i = 0$ for all $i \in [L]$,
the result remains true for any $\alpha>0$.}
of $|\Gamma u^\star|_i$ for all $i \in [L]$.
Choose $\beta$ such that $\alpha\sigma>\beta>0$.
The sequence $\tilde u^k$ defined in \eqref{algo:diff} converges
to the debiasing $\tilde u^\star$ of $u^\star$.
\end{thm}

Before turning to the proof of this theorem, let us introduce a first lemma.
\newcommand{\ut}{\tilde u}
\begin{lem}\label{lem:sad}
The debiasing $\ut^\star$ of $u^\star$ is the solution of the saddle-point problem
\changemargin{-3pt}{-3pt}{%
\begin{equation}\label{pb:sad}
  \min_{{\ut\in\RR^N}}\max_{\tilde z\in \RR^L}  \|\Phi{\ut}-f\|^2+\left<\Gamma\ut,\tilde z\right>- \iota^{}_{{F_{\Ii}}}(\tilde z),
\end{equation}}%
where $\iota^{}_{F_{\Ii}}$ 
is  the indicator function of
 the convex set ${F^{}_{\Ii}}\!=\!\enscond{p\in \RR^L}{p^{}_{\Ii}\!=\!0}.$
\end{lem}

\begin{proof}
As $\Phi U$ has full column rank, the debiased solution is the unique solution
of the constrained least square estimation problem
\changemargin{-3pt}{-3pt}{%
\begin{equation}\label{pb:ls}
  \ut^\star = U(\Phi U)^+ f = \uargmin{\ut\in\mathcal{U}^\star} \|\Phi \ut-f\|^2~.
\end{equation}}%
Remark that $\ut\in\mathcal{U}^\star\!=\!\Ker[\llic{\Gamma}]
\Leftrightarrow(\Gamma\ut)_{\Ii^c}=0\Leftrightarrow \iota^{}_{F_{\Ii^c}}(\Gamma\ut)=0$, where ${F^{}_{\Ii^c}}\!=\!\enscond{p\in \RR^L}{p^{}_{\Ii^c}\!=\!0}$.

Using Fenchel transform,
$\iota^{}_{F_{\Ii^c}}(\Gamma\ut)\!=\!\max_{\tilde z}\left<\Gamma\ut,\tilde z\right>- \iota^{*}_{{F_{\Ii^c}}}(\tilde z)$,
where $\iota^{*}_{{F_{\Ii^c}}}$ is the convex conjugate of $\iota^{}_{{F_{\Ii^c}}}$.
Observing that $\iota^{}_{F_{\Ii}}\!=\!\iota^{*}_{F_{\Ii^c}}$ concludes the proof.
\hfill $\square$
\end{proof}

Given Lemma \ref{lem:sad}, replacing $\Pi_{z^k+\sigma \Gamma v^k}$ in
\eqref{algo:diff}
by the projection onto ${F_{\Ii}}$, i.e.,
\changemargin{-3pt}{-3pt}{%
\begin{equation}\label{proj_converged}
\Pi_{{F_{\Ii}}}(\tilde z)^{ }_{\Ii^c}=\tilde z^{ } _{\Ii^c}~ \qandq \Pi_{{F_{\Ii}}}(\tilde z)_{\Ii}^{}=0~,
\end{equation}}%
leads to
the primal-dual algorithm of \cite{CP} applied to problem \eqref{pb:sad}
which converges to the debiased estimator $\ut^\star$.
It remains to prove
that the projection $\Pi_{z^k+\sigma \Gamma v^k}$ defined in \eqref{algo:diff}
converges to $\Pi_{{F_{\Ii}}}$ in finite time.

\begin{proof}[Theorem \ref{prop:convergence}]
First consider $i\!\in\!\Ii$, i.e., $|\Gamma u^\star|^{}_i\!>\!0$.
By assumption on $\alpha$, $|\Gamma u^\star|^{}_i\!\ge\!\alpha>0$.
Necessary $z_i^\star\!=\!\lambda \sign(\Gamma u^\star)_i$ in order to maximize
\eqref{eq:l1analysis_dual}.
Hence, $|z^\star+\sigma \Gamma u^\star|^{}_i\!\ge\!\lambda+\sigma\alpha$.
Using the triangle inequality shows that
\changemargin{-2pt}{-2pt}{%
\begin{equation}
\lambda+\sigma\alpha\leq|z^\star+\sigma \Gamma u^\star|^{}_i\leq|z^\star-z^k|^{}_i+\sigma |\Gamma u^\star-\Gamma v^k|^{}_i+|z^k+\sigma \Gamma v^k|^{}_i~.
\end{equation}}%
Choose $\epsilon\!>\!0$ sufficiently small such that $\sigma\alpha-\epsilon(1+\sigma ) \!>\! \beta$.
From the convergence of the primal-dual algorithm of \cite{CP},
the sequence $(z^k,u^k,v^k)$ converges to $(z^\star,u^\star,u^\star)$.
Therefore, for $k$ large enough,
$|z^\star-z^k|^{}_i \!<\! \epsilon$, $|\Gamma u^\star-\Gamma v^k|^{}_i \!<\! \epsilon$, and
\changemargin{-2pt}{-2pt}{%
\begin{equation}
  |z^k+\sigma \Gamma v^k|^{}_i\geq\lambda+\sigma\alpha-\epsilon(1+\sigma ) > \lambda+\beta ~.
\end{equation}}%

Next consider $i\!\in\!\Ii^c$, i.e., $|\Gamma u^\star|_i\!=\!0$, where by definition $|z^\star|_i \!\leq\! \lambda$.
Using again the triangle inequality shows that
\changemargin{-2pt}{-2pt}{%
\begin{equation}
|z^k+\sigma \Gamma v^k|^{}_i\leq|z^k-z^\star|^{}_i+\sigma |\Gamma v^k - \Gamma u^\star|^{}_i+|z^\star|^{}_i~.
\end{equation}}%
Choose $\epsilon\!>\!0$ sufficiently small such that $\epsilon(1+\sigma ) \!<\! \beta$.
As $(z^k,u^k,v^k) \to (z^\star,u^\star,u^\star)$, for $k$ large enough,
$|z^k-z^\star|^{}_i \!<\! \epsilon$, $|\Gamma v^k-\Gamma u^\star|^{}_i \!<\! \epsilon$, and
\changemargin{-2pt}{-2pt}{%
\begin{equation}
|z^k+\sigma \Gamma v^k|^{}_i \!<\! \lambda + \epsilon(1+\sigma ) \leq \lambda + \beta~.
\end{equation}}%
It follows that for $k$ sufficiently large $|z^k+\sigma \Gamma v^k|^{}_i \!\leq\! \lambda+\beta$ if and only if $i\!\in\!\Ii^c$, and hence
$\Pi_{z^k+\sigma K v^k}(\tilde z)\!=\!\Pi_{{F_{\Ii}}}(\tilde z)$.
As a result, all subsequent iterations of \eqref{algo:diff}
will solve \eqref{pb:sad}, and hence from Lemma \ref{lem:sad}
this concludes the proof of the theorem.
\hfill $\square$
\end{proof}

\section{Debiasing other affine estimators}

In most cases, $U^\star$ cannot be computed in reasonable memory load and/or time,
such that Eq.~\eqref{eq:debiasing} cannot be used directly.
However, the directional derivative, i.e., the application of $J^\star$ to a direction
$\delta$, can in general be obtained
with an algorithm of the same complexity as the one providing $u^\star$.
If one can compute the directional derivatives for any direction,
a general iterative algorithm for the computation of $\tilde{u}^\star$ can be
derived as given in Algorithm 1.

The proposed technique relies on the fact that given $n \!=\! \dim(\Mm^\star)$
uniformly random directions
$\delta_1, \ldots, \delta_n$ on the unit sphere of $\RR^P$,
$J^\star \delta_1, \ldots, J^\star \delta_n$ forms a basis of $\Ima[J^\star]$
almost surely.
Given this basis, the debiased solution can so be retrieved from \eqref{eq:debiasing}.
Unfortunately, computing the image of the usually large number $n$
of random directions can be computationally prohibitive.

The idea is to approach the debiased solution
by retrieving only a low dimensional subspace of $\Mm^\star$
leading to a small approximation error.
Our greedy heuristic is to chose random perturbations around the current residual
(the strength of the perturbation being controlled by a parameter $\epsilon$).
As soon as $\epsilon >0$, the algorithm
converges in $n$ iterations as explained above.
But, by focusing in directions guided by the current residual,
the algorithm refines in priority the directions
for which the current debiasing gets significantly away from the data $f$, i.e.,
directions that encodes potential remaining bias.
Hence, the debiasing can be very effective even though
a small number of such directions has been explored.
We notice in our experiments that with a small value of $\epsilon$,
this strategy leads indeed to a satisfying debiasing,
close to convergence, reached in a few iterations.

\paragraph{The nonlocal-means example.}
The block-wise nonlocal-means proposed in \cite{buades2005nlmeans}
can be rewritten as an instance of
the minimization problem \eqref{eq:variationnal_model}
with
\changemargin{-4.5pt}{-4.5pt}{%
\begin{equation}
  E(u, f) = \frac{1}{2} \sum_{i,j} w_{i,j} \norm{\Pp_i u - \Pp_j f}^2
  \qwithq
  w_{i,j} = \phi\left(\frac{\norm{\Pp_i f - \Pp_j f}^2}{2 \sigma^2}\right)
\end{equation}}%
where $i \in [n_1] \times [n_2]$ spans the whole image domain,
$j-i \in [ -s, s] \times [-s, s]$ spans a limited search window domain
and $\sigma^2$ is the noise variance.
We denote by $\Pp_i$ the linear operator extracting the patch at pixel $i$ of size $(2p + 1) \times (2p + 1)$.
Note that we assume periodical conditions such that all quantities
remain inside the image domain.
The kernel $\phi : \RR^+ \to [0, 1]$ is a decreasing function
which is typically a decay exponential function.
Taking $\phi$ piece-wise constant%
\footnote{For instance, by quantification
  on a subset of predefined values in $[0, 1]$.}%
, leads to
computing $u^\star$ and its Jacobian at $f$ applied to $\delta$ for almost all $f$
as follow
\changemargin{-4.5pt}{-4.5pt}{%
\begin{equation}
  u^\star_i = \frac{\sum_j \bar{w}_{i,j} f_j}{\sum_j \bar{w}_{i,j}}
  \qandq
  (J^\star \delta)_i =
  \frac{\sum \bar{w}_{i,j} \delta_j}{\sum_j \bar{w}_{i,j}}
  \qwithq
  \bar{w}_{i,j} = \sum_k w_{i-k, j-k}
\end{equation}}%
where $k \in [-p, p] \times [-p, p]$ spans the patch domain.
Note that the values of $w$ and $\bar{w}$ can be obtained by
discrete convolutions leading to an algorithm with complexity in $O(N s^2)$,
independent of the half patch size $p$%
\ifthenelse{\boolean{arxiv}}{ (see Appendix \ref{sec:algo_nlmeans}).}{.}

With such a choice of $\phi$, the block-wise nonlocal filter becomes a
piece-wise affine mapping of $f$ and hence Algorithm 1
applies.\\[-1.5em]

\begin{figure}[!t]
\centering
\small
\vspace{-0.5em}
\begin{minipage}{\linewidth}
\rule{\linewidth}{2px}
\begin{algorithmic}[0]
  \State \textbf{Algorithm 1} General debiasing pseudo-algorithm for the computation of $\tilde{u}^\star$.
\end{algorithmic}
\vspace{-0.2cm}
\rule{\linewidth}{1px}
\begin{algorithmic}[0]
  \State \makebox[1.5cm][l]{\textbf{Inputs:}} $f \in \RR^P$, $u^\star \in \RR^N$, $\delta \in \RR^P \to J^\star \delta \in \RR^N$, $\epsilon>0$.
  \State \makebox[1.5cm][l]{\textbf{Outputs:}} $\tilde{u}^\star \in \RR^N$ and $U \in \RR^{N \times n'}$ with $n' \leq n$ an orthonormal family of $\Ima[J^\star]$
   \State \vspace{-0.6em}
  \State \makebox[1.5cm][l]{Initialize} $U \leftarrow [\;]$
  \Repeat { until $\tilde{u}^\star$ reaches convergence}
  \State \makebox[1.5cm][l]{Generate} $\delta \leftarrow \eta/\norm{\eta}, \eta \sim \Nn_P(0, \Id)$
  \hfill (perturbation ensuring convergence)
  \State \makebox[1.5cm][l]{Compute} $\displaystyle u' \leftarrow J^\star (f - \Phi \tilde{u}^\star + \epsilon \delta)$
  \hfill (perturbed image of the current residual)
  \State \makebox[1.5cm][l]{Compute} $\displaystyle e \leftarrow u' - U (U^t u')$
  \hfill (projection~on~the~orthogonal~of~the current~$\Mm^\star$)
  \State \makebox[1.5cm][l]{Update} $U \leftarrow [ U \; e/\norm{e} ]$
  \State \makebox[1.5cm][l]{Update} $\tilde{u}^\star \leftarrow u^\star + U ((\Phi U)^+ (f - \Phi u^\star))$
  \color{white} \Until\vspace{-1.4em}
\end{algorithmic}
\vspace{-0.5em}
\rule{\linewidth}{1px}
\vspace{-1.5em}
\end{minipage}
\vspace{1em}
\end{figure}


\begin{figure}[!t]
  \centering
  \vspace{-0.5em}
  \includegraphics[width=0.8\linewidth]{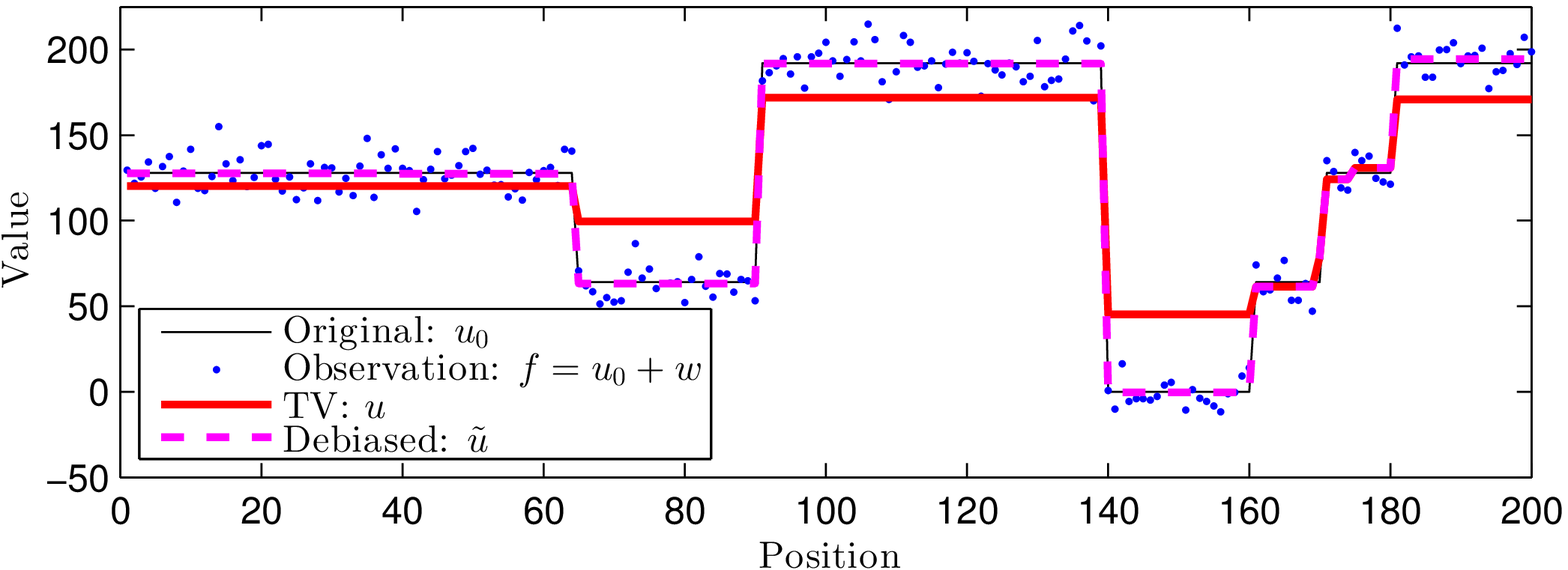}
  \vspace{-1em}
  \caption{%
    Solutions of 1D-TV and our debiasing on a piece-wise constant signal.
  }
  \vspace{1em}
  \label{fig:tv1d}
  \centering
  \begin{minipage}[b]{0.32\linewidth}%
    \includegraphics[width=\linewidth,viewport=256 284 512 474,clip]{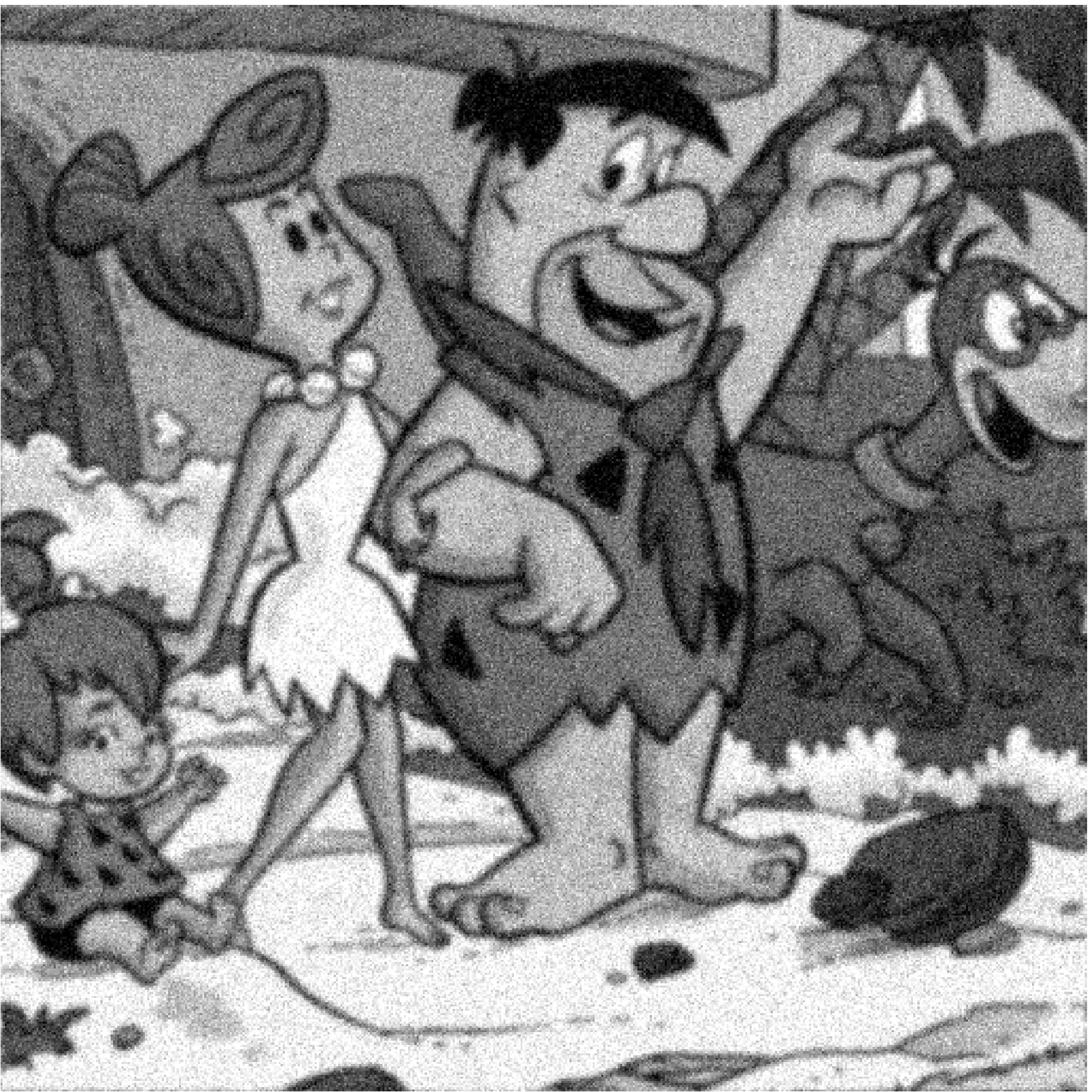}%
    \hspace{-\linewidth}%
    \begin{minipage}{\linewidth}%
      \vspace{1.7em}%
      \centering%
      \color{black} \tiny \textbf{ PSNR 19.13 / SSIM 0.76 }%
    \end{minipage}%
  \end{minipage}%
  \hfill%
  \begin{minipage}[b]{0.32\linewidth}%
    \includegraphics[width=\linewidth,viewport=256 284 512 474,clip]{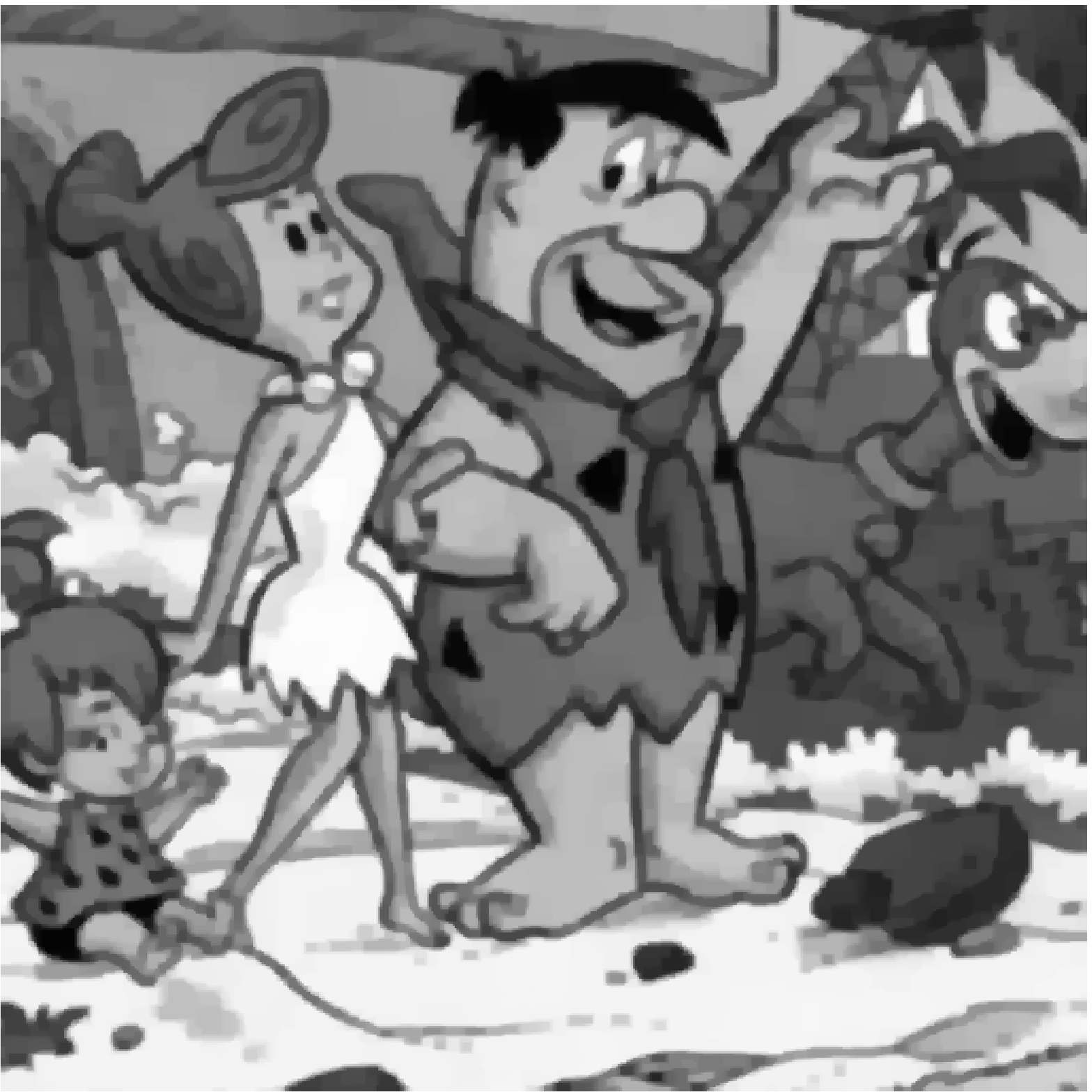}%
    \hspace{-\linewidth}%
    \begin{minipage}{\linewidth}%
      \vspace{1.7em}%
      \centering%
      \color{black} \tiny \textbf{ PSNR 20.61 / SSIM 0.80 }%
    \end{minipage}%
  \end{minipage}%
  \hfill%
  \begin{minipage}[b]{0.32\linewidth}%
    \includegraphics[width=\linewidth,viewport=256 284 512 474,clip]{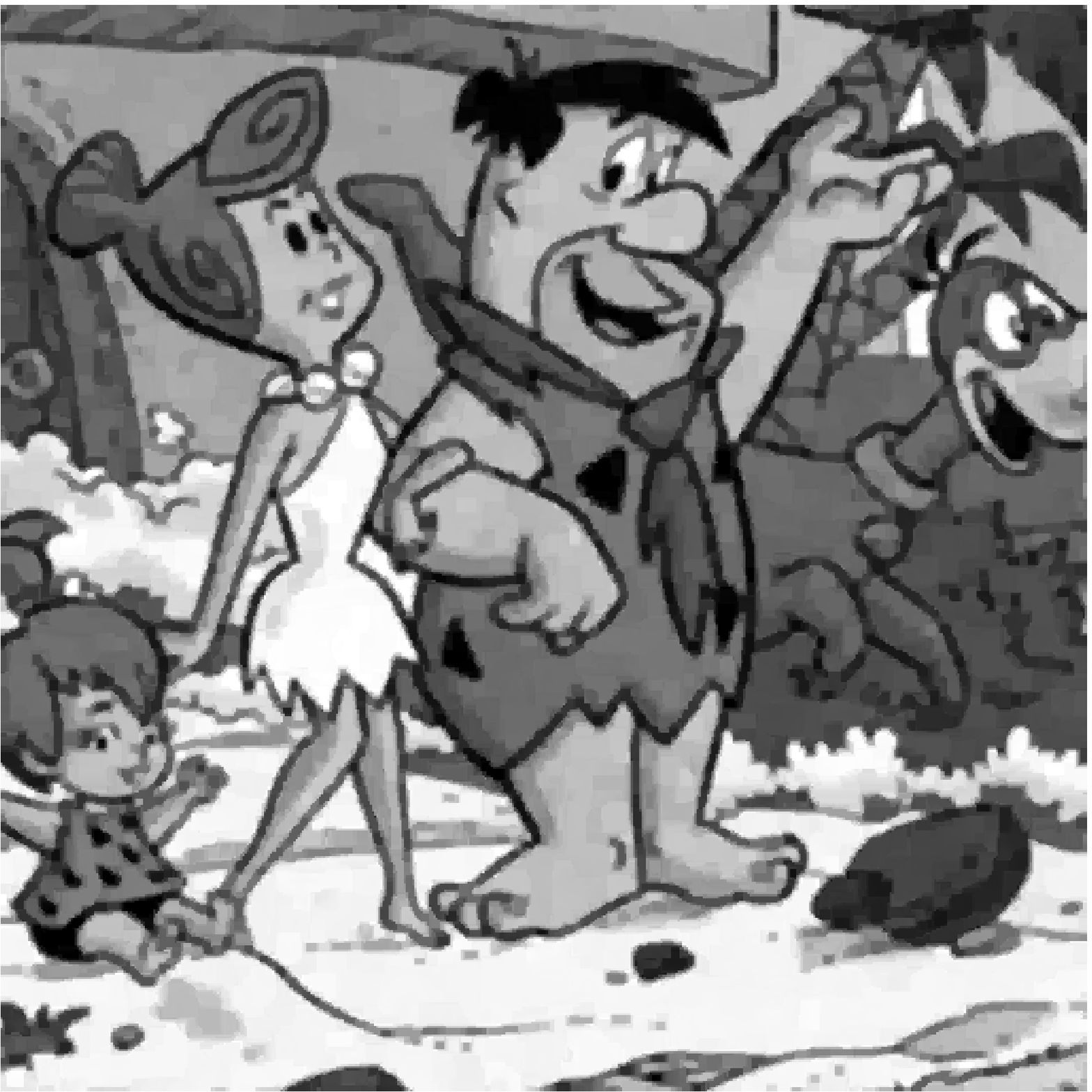}%
    \hspace{-\linewidth}%
    \begin{minipage}{\linewidth}%
      \vspace{1.7em}%
      \centering%
      \color{black} \tiny \textbf{ PSNR 21.90 / SSIM 0.87 }%
    \end{minipage}%
  \end{minipage}%
  \vspace{-0.5em}
  \caption{%
    (left) Blurry image $f \!=\! \Phi u_0 \!+\! w$,
    (center) TV $u^\star$,
    (right) debiased $\tilde{u}^\star$.
  }
  \vspace{1em}
  \label{fig:tv2d_deconv}
  \begin{minipage}[b]{0.32\linewidth}%
    \includegraphics[width=\linewidth,viewport=0 46 256 236,clip]{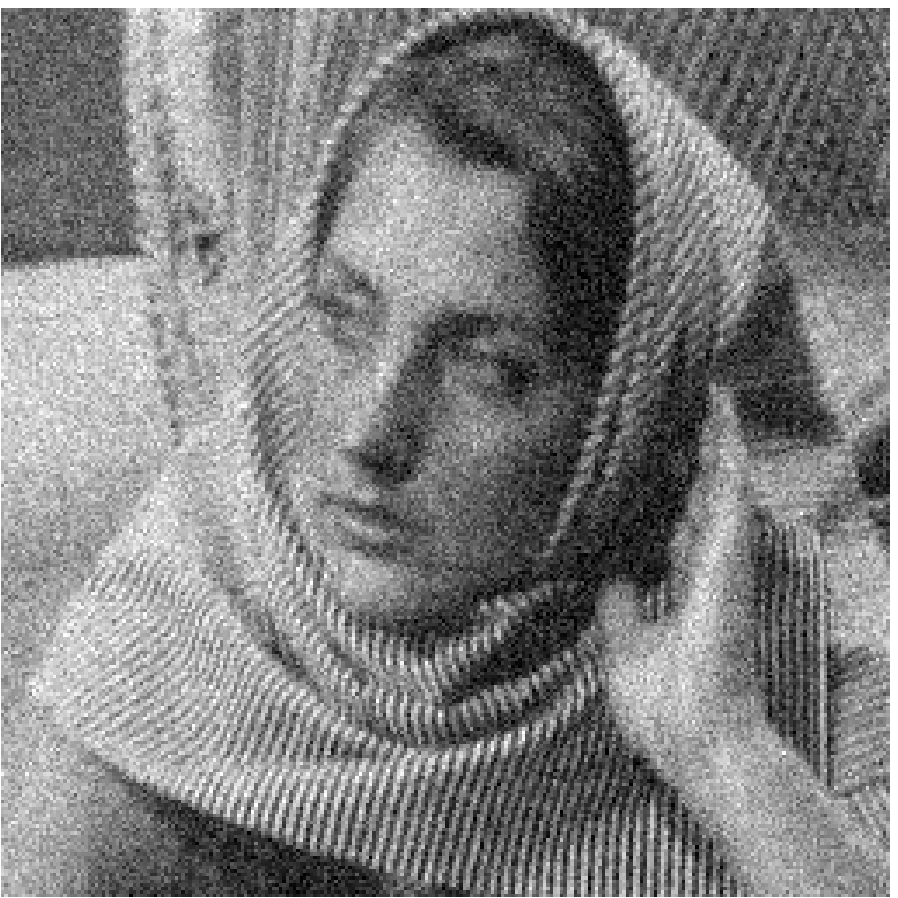}%
    \hspace{-\linewidth}%
    \begin{minipage}{\linewidth}%
      \vspace{1.7em}%
      \centering%
      \color{black} \tiny \textbf{ PSNR 22.14 / SSIM 0.52 }%
    \end{minipage}%
  \end{minipage}%
  \hfill%
  \begin{minipage}[b]{0.32\linewidth}%
    \includegraphics[width=\linewidth,viewport=0 46 256 236,clip]{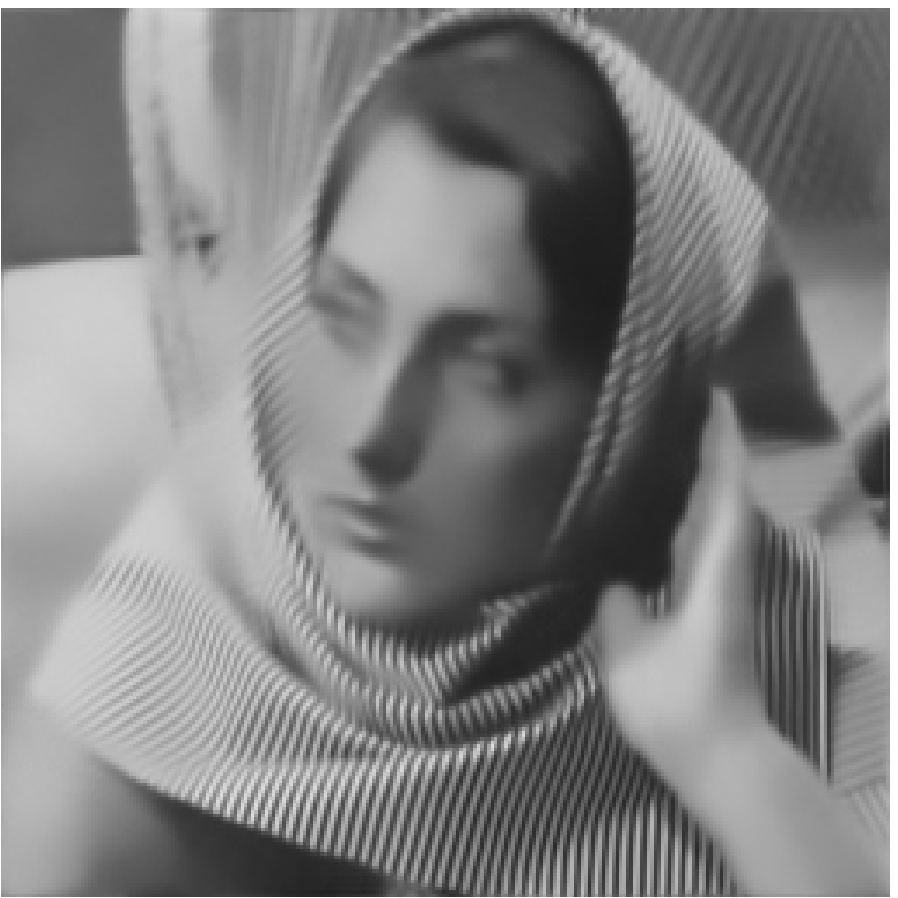}%
    \hspace{-\linewidth}%
    \begin{minipage}{\linewidth}%
      \vspace{1.7em}%
      \centering%
      \color{black} \tiny \textbf{ PSNR 27.89 / SSIM 0.82 }%
    \end{minipage}%
  \end{minipage}%
  \hfill%
  \begin{minipage}[b]{0.32\linewidth}%
    \includegraphics[width=\linewidth,viewport=0 46 256 236,clip]{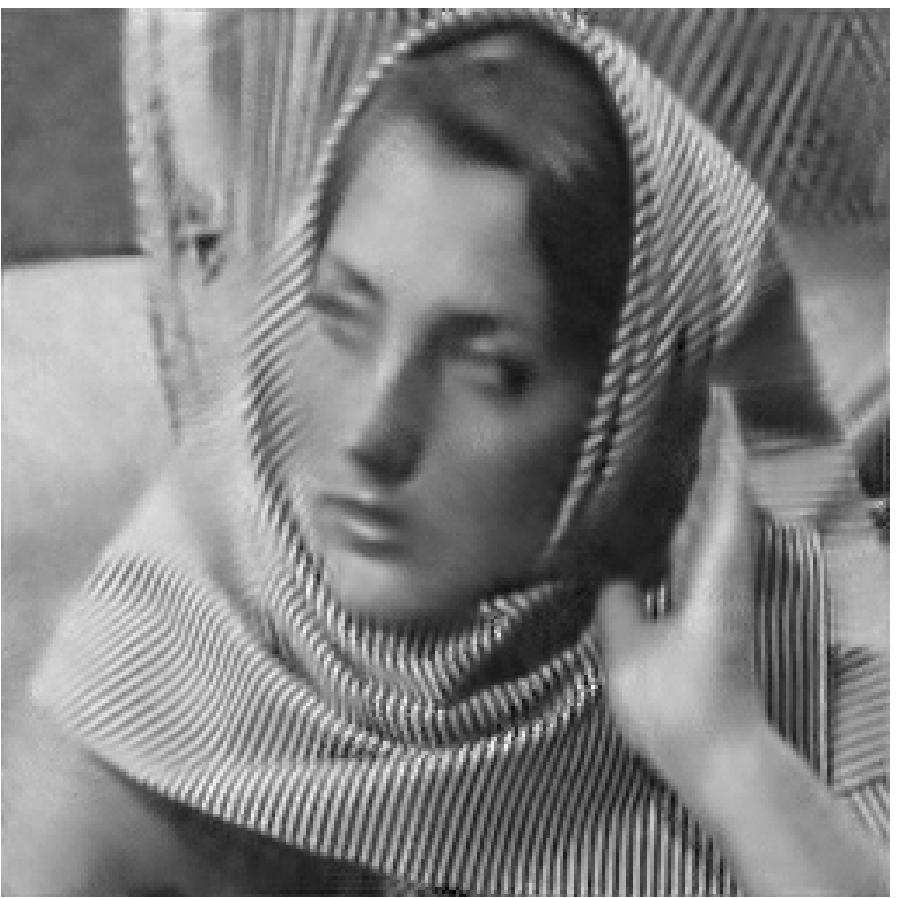}%
    \hspace{-\linewidth}%
    \begin{minipage}{\linewidth}%
      \vspace{1.7em}%
      \centering%
      \color{black} \tiny \textbf{ PSNR 29.17 / SSIM 0.87 }%
    \end{minipage}%
  \end{minipage}%
  \vspace{-0.5em}
  \caption{%
    (left) Noisy image $f \!=\! u_0 \!+\! w$,
    (center) nonlocal-means $u^\star$,
    (right) debiased $\tilde{u}^\star$.
  }
  \label{fig:nlm}
  \vspace{1em}
\end{figure}

\section{Numerical experiments and results}

Figure \ref{fig:tv1d} gives an illustration of
TV used for denoising a 1D piece-wise constant signal
in $[0, 192]$ and
damaged by additive white Gaussian noise (AWGN) with a standard deviation $\sigma\!=\!10$.
Even though TV has perfectly retrieved the
support of $\nabla u_0$ with one more extra jump, the intensities of some regions
are biased.
Our debiasing is as expected unbiased for every region.

Figure \ref{fig:tv2d_deconv} gives
an illustration of our debiasing of
2D anisotropic TV used for the restoration of
an $8bits$ approximately piece-wise constant image damaged by
AWGN with $\sigma\!=\!20$.
The observation operator $\Phi$ is a Gaussian convolution kernel of
bandwidth $2$px.
TV introduced a significant loss of contrast, typically for
the thin contours of the drawing, which are re-enhanced
by our debiased result.

Figure \ref{fig:nlm} gives an illustration of our iterative
debiasing for the block-wise nonlocal-means algorithm used
in a denoising problem for an $8bits$ image
enjoying many repetitive patterns
and damaged by AWGN with $\sigma=20$.
Convergence has been considered as reached after $4$ iterations only.
Our debiasing provides favorable results with many enhanced details
compared to the biased result.%
\ifthenelse{\boolean{arxiv}}{\\

Please refer to Appendix \ref{sec:supp_experiments}
for more details and experiments.
}{}



\section{Conclusion}

We have introduced in this paper a mathematical definition of debiasing which has led
to an effective debiasing technique that can remove
the method bias that does not arise from the unavoidable choice of the model.
This debiasing technique simply consists in applying a least-square estimation
constrained to the model subspace chosen implicitly by the
original biased algorithm.
Numerical experiments have demonstrated the efficiency of our technique in retrieving
the correct intensities while respecting the structure
of the original model subspace.
Our technique is nevertheless limited to locally affine estimators.
Isotropic total variation, structured sparsity or nonlocal-means with smooth kernels
are  not yet handled by our debiasing technique, and left for future work.

\bibliographystyle{abbrv}

\end{document}